\documentclass[10pt]{amsart}

\usepackage{amsgen,amsthm,amstext,amsbsy,amsopn,amsfonts,amssymb, enumerate}

\usepackage{amsmath}

\newcommand\la{\langle}
\newcommand\ra{\rangle}

\newcommand\hh{{\mathfrak h}}

\newcommand\nn{{\mathfrak n}}

\newcommand\vv{{\mathfrak v}}

\newcommand\ww{{\mathfrak w}}
\newcommand\zz{{\mathfrak z}}

\newcommand\iso{{\mathfrak{iso}}}

\newcommand\RR{\mathbb R}

\newcommand\so{{\mathrm{so}}}

\newcommand\ad{\operatorname{ad}}

\newcommand\Iso{\operatorname{Iso}}

\theoremstyle{plain}

\newtheorem{thm}{Theorem}[section]
\newtheorem{lem}[thm]{Lemma}
\newtheorem{prop}[thm]{Proposition}

\theoremstyle{definition}
\newtheorem{defn}[thm]{Definition}
\newtheorem{rem}[thm]{Remark}
\newtheorem{example}[thm]{Example}

\begin{document}

\title[Magnetic trajectories on 2-step nilmanifolds]
{Magnetic trajectories on 2-step nilmanifolds}

\author{Gabriela P. Ovando, Mauro Subils}

\thanks{{\it (2000) Mathematics Subject Classification}: 70G45, 22E25, 53C99, 70G65 }


\thanks{Partially supported by  SCyT (UNR)}

\address{ Departamento de Matem\'atica, ECEN - FCEIA, Universidad Nacional de Rosario.   Pellegrini 250, 2000 Rosario, Santa Fe, Argentina.}

\

\email{gabriela@fceia.unr.edu.ar}

\email{subils@fceia.unr.edu.ar}


\begin{abstract}  The aim of this work is the study of magnetic trajectories on nilmanifolds. The magnetic equation is written and the corresponding solutions are found for a family of invariant Lorentz forces on a 2-step nilpotent Lie group equipped with a left-invariant metric.  Some examples are computed in the Heisenberg Lie groups $H_n$ for $n=3,5$, showing  differences with the case of exact forms. Interesting  magnetic trajectories related to elliptic integrals appear in $H_3$. The question of existence of closed or periodic magnetic trajectories for every energy level on Lie groups or  on compact quotients is treated. 
\end{abstract}

\maketitle

 \noindent\section{Introduction}
From the mechanical perspective, geodesics describe the motion
of particles that are not experiencing any force.  In the presence of a force, known as {\it  Lorentz force}, the  behavior of a charged particle is described by an equation of the form:
\begin{equation}\label{mageq}
	\nabla_{\gamma'}{\gamma'}= q F\gamma'
\end{equation}
where $\gamma$ is a curve on a Riemannian manifold $(M, g)$,  $\nabla$ is the corresponding Levi-Civita connection and $F$ is a skew-symmetric $(1,1)$-tensor such that  the corresponding 2-form is closed. This example is taken from electromagnetism theory, and different examples arise associated to other geometries,  for instance the  potential gauge
on a  principal circle bundle $P(M^n, \mathbb S^1)$, the  K\"ahler (uniform) magnetic field on a 
K\"ahler manifold (see for instance \cite{Ad}), or the contact magnetic field on a Sasakian manifold \cite{CF}. The main purpose is to solve the equation  and to study the underlying geometries. From the dynamical perspective,  many authors consider the  topological entropy
or the Anosov property of magnetic flows (see for instance \cite{BP}).

The  magnetic trajectories are quite different from geodesics. In spaces of dimension two \cite{Co,Su} where the magnetic field is given by $q dA$, being $dA$ the area element, the magnetic trajectories follow: on the Euclidean plane $\RR^2$ they are circles, on the sphere $S^2$ they are small circles, and on the hyperbolic plane  the trajectories are either closed when $|q|\geq 1$, or open in the rest. These results were extended in different directions. For example, trajectories corresponding to magnetic fields defined as scalar multiples of the K\"ahler form on a K\"ahler manifold were studied in \cite{Ad1,Ad2}, on an almost K\"ahler manifold in  \cite{EI}. Different kind of spaces were considered, such as the sphere \cite{Sc}, the torus \cite{Ta}, and those related to Lie groups and their actions, for instance in \cite{BJ1, BJ2}.

For 2-step nilmanifolds, the topic  was recently considered by Epstein, Gornet and Mast in \cite{EGM}. They studied periodic magnetic geodesics on Heisenberg manifolds,  obtaining  an  analysis of left-invariant exact magnetic flows. They search  closed trajectories  and the corresponding  energy level where they occur, and they determine the Ma\~n\'e critical value. Indeed there are previous and important works describing closed geodesics on 2-step nilmanifolds such as \cite{dC, dM,LP,Ma}. 

In this paper we concentrate the attention to 2-step nilpotent Lie groups equipped with a left-invariant metric and their associated compact quotients. 
We have two main   purposes:
 \begin{enumerate}[(i)]
	\item  to describe the solutions of the magnetic equation; 
	\item to determine closeness conditions of magnetic trajectories on compact quotients. 
\end{enumerate} 
Indeed  the first goal ask for solving a differential equation, the second one connects with dynamical questions. 

After the Euclidean spaces, the 2-step nilpotent Lie groups, have a simple algebraic structure but develop a very interesting geometry  to study. They  are the counterpart of 2-step Lie algebras, that satisfy $[[U,V], W]=0$ for all $U,V,W$ in the Lie algebra, namely $\nn$. When equipped with a metric, such Lie algebra $\nn$  admits an orthogonal decomposition 
$$\nn=\vv \oplus \zz, \quad \mbox{ where } \vv=\zz^{\perp}, \quad (*)$$
and $\zz$ denotes the center of the Lie algebra. This decomposition is deeply related to the geometry of the corresponding 2-step nilpotent Lie group $(N, g)$ whenever $g$ is a left-invariant metric \cite{Eb}. 

We consider a Lorentz force  $F$ on  $(N,g)$ which is invariant by translations on the left. The first step is to write clearly the magnetic equation in this situation in terms of the decomposition (*). This is done by making use of the exponential map. 

At the corresponding Lie algebra level, there is a magnetic equation for curves on $\nn$. We also obtain a description of the magnetic trajectories for Lorentz forces corresponding to exact forms (Theorem 3.3),   by identifying them  in the set of all differentiable curves on $\nn$.

For the case of left-invariant Lorentz forces on the Lie group $(N, g)$ preserving the decomposition (*), that is $F\zz \subset \zz $ and $F\vv \subset \vv$, we obtain the explicit formula for the magnetic trajectories (Theorem \ref{magnetictrayF0}).  



In another work, the authors study left-invariant 2-forms on 2-step nilpotent nilmanifolds obtaining the next obstruction: 
{\it If the 2-step nilpotent Lie algebra $(\nn, \la\,,\, \ra)$ is non-singular and its dimension satisfies: $\dim \nn > 3 \dim \zz$, then any skew-symmetric map $F:\nn \to \nn$ giving rise to a magnetic field,  preserves the decomposition (*).} 

\smallskip

This means that Theorem \ref{magnetictrayF0} gives all possible  magnetic trajectories in Lie groups satisfying these two conditions. 







Next we work out some examples. In the Heisenberg Lie group of dimension 
three, we study Lorentz forces that interchange the subspaces $\vv$ and $\zz$. In this situation we prove that solutions of the magnetic equation  are related to elliptic integrals. Indeed, those solutions are very different from the ones obtained in Theorem \ref{magnetictrayF0}, see also Example 3.9.  But according to  Greenhill in  \cite{Gr},  elliptic integrals were applied in electromagnetism theory in the 19th century, a fact known by Legendre and other mathematicians at that time. 

On the induced compact nilmanifolds $\Gamma \backslash H_3$, we show conditions for periodicity of magnetic trajectories. The key is the following result obtained for 2-step nilpotent Lie groups

\smallskip

Lemma \ref{lambdaper}: {\em Let $\lambda=\exp(W_1+Z_1)$ be any element in the 2-step nilpotent Lie group $(N, \la\,,\,\ra)$. If a left-invariant Lorentz force $F$ of type II admits a $\lambda$-periodic trajectory then $W_1\in Ker F$.}

\smallskip

Finally, we consider a left-invariant Lorentz force in the Heisenberg Lie group of dimension five that preserves the decomposition $\hh_5=\vv \oplus \zz$. We study the existence of periodic magnetic trajectories and we find out differences to the results in \cite{EGM}. In our case, there exist closed magnetic geodesics for every energy level, whenever the magnetic field is non-exact. 


The paper is organized as follows: in the first section we recall the basic notions of 2-step nilpotent (real) Lie groups equipped with a left-invariant metric. 
 In the next section we study the magnetic equation and properties of solutions. We find  solutions for any invariant Lorentz force preserving the decomposition \ref{decomp2}. 
The final section shows the explicit examples we mentioned above: magnetic trajectories related to elliptic integrals on $H_3$ and the question of periodicity on compact quotients $\Lambda\backslash H_3$. Later, magnetic trajectories on the Heisenberg Lie group $H_5$ are analysed.

\section{Lie groups of step two with a left-invariant metric}\label{general}

A Lie group is called 2-step nilpotent if its Lie algebra is 2-step nilpotent, that is, the Lie bracket satisfies $[[U,V], W]=0$ for all $U,V,W\in \nn$. Throughout this paper Lie groups, so as their Lie algebras are considered over $\RR$. 

Whenever $N$ is simply connected, the Lie group  $N$ is diffeomorphic to $\RR^n$ and  the exponential map is a diffeomorphism so that  the multiplication map on $N$ obeys the following relation
$$\exp (V) \exp(W)= \exp (V +W+\frac12 [V,W]), \quad \mbox{ for all } U,V, W \in \nn.$$


\begin{example} \label{exa1} The smallest dimensional 2-step nilpotent Lie group is the Heisenberg Lie group $H_3$. It has dimension three and its Lie algebra is spanned by vectors $e_1, e_2, e_3$ satisfying the non-trivial Lie bracket relation
	$$[e_1,e_2]=e_3.$$
	The Lie group $H_3$ can be modeled on $\mathbb R^3$ equipped with the product operation given by
	$$(v_1,z_1)(v_2,z_2)=(v_1+v_2, z_1+z_2+\frac{1}2 v_1^tJv_2),$$
	where $v_i=(x_i,y_i)$, i=1,2 and $J:\RR^2 \to \RR^2$ is the linear map $J(x,y)=(y, -x)$. By using this, usual computations show that a basis of left-invariant vector fields is given at $p=(x,y,z)$ by
	$$e_1(p)=\partial_x -\frac12 y \partial_z, \quad e_2(p)=\partial_y+\frac12 x \partial_z, \quad e_3(p)=\partial_z.$$
Another presentation of the Heisenberg Lie group is given by $3\times3$-triangular real  matrices with 1's on the diagonal with the usual multiplication of matrices.  
	\end{example} 

A Riemannian  metric $\la\,,\,\ra$ on the Lie group  $N$ is called {\it left-invariant} if  translations on the left by elements of the group are isometries. Thus, a left-invariant metric  is determined at the Lie algebra level $\nn$, usually identified with the tangent space at the identity element $T_eN$. The metric on the Lie algebra, also denoted $\la\,,\,\ra$, determines an orthogonal decomposition as vector spaces on the Lie algebra: 
\begin{equation}\label{decomp2}
	\nn=\vv \oplus \zz, \quad \mbox{ where }\quad \vv =\zz^{\perp}
\end{equation}
and $\zz$ denotes the center of $\nn$. The subspaces $\vv$ and $\zz$ induce left-invariant distributions on $N$, denoted by $\mathcal V$ and $\mathcal Z$.  

The decomposition in Equation \eqref{decomp2} induces the skew-symmetric maps $j(Z):\vv \to \vv$, for every $Z\in \zz$,  implicitly defined by 
\begin{equation}\label{j}
	\la Z, [V,W]\ra =\la j(Z) V, W \ra \qquad \mbox{ for all } Z\in \zz, V, W\in \vv. 
\end{equation}

Note that $j:\zz \to \mathfrak{so}(\vv)$ is a linear map. Let $C(\nn)$ denote the commutator of the Lie algebra $\nn$. One has the splitting 
$$\zz=C(\nn)\oplus \ker(j)$$ as orthogonal direct sum of vector spaces. In fact, 
\begin{itemize}
	\item since $\la Z, [U,V]\ra=0$ for all $U,V\in \vv$ and $Z\in \ker(j)$,  then $\ker(j)\perp C(\nn)$.
	\item $\dim \zz = \dim \ker(j)+ \dim C(\nn)$. 
	\item The restriction $j:C(\nn) \quad \mapsto \quad \mathfrak{so}(\vv)\quad \mbox{is injective}.$
\end{itemize} 
In fact, assume $j(Z)=j(\bar{Z})$ for $Z, \bar{Z}\in C(\nn)$. Then $j(Z-\bar{Z})=0$, so that $Z-\bar{Z}\in \ker(j)\cap C(\nn)$. Thus $Z-\bar{Z}=0$. 

See the proof of the next result in Proposition 2.7 in \cite{Eb}. 
 
 \begin{prop} \cite{Eb} Let $(N, \la,,\,\ra)$ denote a 2-step nilpotent Lie group with a left-invariant metric. Then 
 	\begin{itemize}
 		\item the subspaces $\ker j$ and $C(\nn)$ are commuting ideals in $\nn$. 
 		\item Let $E=\exp(\ker(j))$. Then $E$ is the Euclidean de Rham factor of $N$ and $N$ is isometric to the Riemannian product of the totally geodesic submanifolds $E$ and $\bar{N}$ where $\bar{N}=\exp(\vv \oplus C(\nn))$. 
 	\end{itemize}
 	
 	\end{prop}

 \begin{example}
 	Let $\hh_3$ denote the Heisenberg Lie algebra of dimension three with basis $e_1, e_2, e_3$ as in Example \ref{exa1}. Take the metric so that  this basis is orthonormal. It is not hard to see that the center is the subspace
 	 $\zz=span\{e_3\}$, while its orthogonal complement is  the subspace $\vv=span\{e_1, e_2\}$ and moreover the map $j:\zz\to \so(\vv)$ is generated by 
 	$$j(e_3)=\left( \begin{matrix}
 		0 & -1\\
 		1 & 0 
 	\end{matrix}\right), 
 $$
 in the basis $e_1, e_2$ of $\vv$. 
 \end{example}

Let $\nabla$ denote the Levi-Civita connection on the 2-step Lie group  $(N, \la\,,\,\ra)$. Since the metric is invariant by left-translations, for $X, Y$ left-invariant vector fields one has the following formula for the covariant derivative:
$$\nabla_X Y = \frac12 \{[X,Y]- \ad(X)^*(Y)-\ad(Y)^*(X)\}$$
where $\ad(X)^*, \ad(Y)^* $ denote the adjoints of $\ad(X), \ad(Y)$ respectively.  Thus, one obtains
$$\left\{ 
\begin{array}{lll}
	(i) & \nabla_Z \widetilde{Z}=0 & \mbox{ for all } Z, \widetilde{Z}\in \zz,\\
	(ii) & \nabla_Z X= \nabla_X Z = -\frac12 j(Z) X & \mbox{ for all } Z\in \zz, X\in \vv,\\
	(iii) & \nabla_X \widetilde{X}=\frac12 [X,\widetilde{X}] & \mbox{ for all } X, \widetilde{X}\in \vv.
\end{array}
\right.
$$

Furthermore the isometry group of the nilpotent Lie group $(N, \la\,,\,\ra)$ is the semidirect product
$$\Iso(N)=H \ltimes N, \qquad\mbox{ where }{\begin{array}{ll}
		N &\mbox{is the nilradical of } \Iso(N) \mbox{ and } \\
		H  & \mbox{is the  group of orthogonal automorphisms,}
\end{array}}$$
result that was proved in 1963 by Wolf  \cite{Wo} (see also \cite{Wi}). The action of the subgroup $H$ on the ideal $N$ is given by $f \cdot L_n = L_{f(n)}$ for every left translation $L_n$ with $n\in N$, and every $f\in H$.

Whenever $N$ is simply connected, the set of  orthogonal automorphisms on $H$ is in correspondence with the set of orthogonal authomorphisms of $\nn$. 
Thus, the  Lie algebra of the isometry group, denoted by $\iso(\nn)$, is the direct sum as vector spaces $\iso(\nn)=\hh \ltimes \nn$, where $\hh$ is the Lie  subalgebra of $H$ and $\nn$ an ideal being the Lie algebra of $N$. The  Lie subalgebra $\hh$  consists of skew-symmetric derivations,
$$\hh=\{ D : \nn \to \nn \, : \, D[X, Y]=[DX, Y] + [X, DY] \mbox{ and } \la DX, Y\ra + \la X, DY\ra =0\}.$$

\begin{example}\label{ortAutHeis} In the case of the Heisenberg Lie algebra equipped with its canonical metric as above, it is easy to prove that any skew-symmetric derivation has a matrix in the basis $e_1, e_2, e_3$ of the form
	$$\left( \begin{matrix}
		0 & -b & 0\\
		b & 0 & 0\\
		0 & 0 & 0
	\end{matrix}\right), \qquad \quad b\in \RR.
	$$
	Thus the isometry Lie algebra $\iso(\hh_3)=\RR \oplus \hh_3$ where any  $b \in \RR$ corresponds to a derivation as above. Furthermore, the Lie algebra $\iso(\hh_3)$ is isomorphic to the oscillator Lie algebra of dimension four.  The isometry group of the Heisenberg Lie group $(H_3, \la\,,\,\ra)$ is  $\Iso(H_3)=\mathrm O(2)\ltimes H_3$, where the action is explicitly given by
	$$(A, (\tilde{V},\tilde{Z}))\cdot (V,Z)= (A(\tilde{V}+V), det(A)(\tilde{Z}+Z+\frac12 [\tilde{V}, V])), \ \  \mbox{ for }A\in \mathrm O(2), \tilde{V}, V\in \vv, \tilde{Z}, Z\in \zz.$$
	\end{example} 
	  
\begin{defn}
A  2 -step nilpotent real Lie algebra $\nn$ with center $\zz$ is called {\em non-singular}  if  $\ad(X): \mathfrak{n} \rightarrow \mathfrak{z}$ is onto for any $X \notin \mathfrak{z}$ \cite{Eb}. The corresponding 2-step nilpotent Lie group will be called non-singular. 

\end{defn}

See the next examples of non-singular Lie algebras. 

\begin{example} {\it Heisenberg Lie algebras.} \label{ExHeis} Let $n\geq 1$ be any integer and let $X_1,Y_1, X_2, Y_2,$ $ \hdots, X_n, Y_n$ be any basis of a real vector space $\vv$ isomorphic to $\RR^{2n}$. Let $Z$ be an element generating a one dimensional space $\zz$. Define a Lie bracket by $[X_i, Y_i]=-[Y_i, X_i]=Z$ and the other Lie brackets by zero. The Lie algebra $\hh_{2n+1}=\vv \oplus \zz$ is the $(2n+1)$-dimensional Heisenberg Lie algebra.   	
\end{example}

\begin{example} {\it Quaternionic  Heisenberg  Lie algebras.} Let $n\geq 1$ be any integer. For each integer $1\leq i \leq n$,  let  $\mathbb H^i$ be a  four dimensional real vector spaces with basis $X_i,Y_i, V_i, W_i$. Let $\zz$ denote a three dimensional real vector space with basis $Z_1, Z_2, Z_3$. Consider the vector space direct sum $\nn=\vv \oplus \zz$, where $\vv=\bigoplus_i \mathbb H^i$. Define a Lie bracket on the Lie algebra $\nn$, $[\,,\,]$, that is $\RR$-bilinear and skew-symmetric with non-trivial relations as follows:
	
	\smallskip
	
	$[Z, \xi]=0$ for all $Z\in \zz, \xi\in \nn$, 
	$$	
	\begin{array}{llrl} 
		[X_i,Y_i]=Z_1 & [X_i,V_i]=Z_2 & [X_i, W_i]=Z_3 & \mbox{ for } 	1\leq i \leq n, \\
		{ [V_i, W_i]=Z_1} & [Y_i, W_i]=-Z_2 & [Y_i, V_i]=Z_3 & \mbox{ for } 	1\leq i \leq n. 	
	\end{array}
	$$
	The resulting Lie algebra is called the quaternionic Heisenberg Lie algebra of dimension $4n+3$. 
\end{example}

Once the 2-step nilpotent Lie algebra $\nn$ is equipped  with a metric, one has the corresponding maps $j(Z):\vv \to \vv$ defined in Equation \eqref{j}. The non-singularity property is equivalent to the condition that  any map $j(Z):\vv\to\vv$ is non-singular for every $Z\in\zz$. And this condition is independent of the metric.

 Non-singular  Lie algebras are also known as {\em fat} algebras because they are the symbol of fat distributions (see \cite{M}). 

More generally, a 2-step nilpotent Lie algebra $\nn$ is said
\begin{itemize}
	\item {\em almost non-singular} if there are elements $Z, \widetilde{Z}\in \zz$ such that $j(Z)$ is non-singular but $j(\widetilde{Z})$ is singular. 
	\item  {\em singular} if any map $j(Z)$ is singular for every $Z\in \zz$.
\end{itemize} 

A family of non-singular Lie algebras is provided by H-type Lie algebras, which are defined as follows. 

Let $(\nn, \la\,,\,\ra)$ denote a 2-step nilpotent Lie algebra equipped with a metric. If the map $j(Z):\vv\to\vv$ is orthogonal for every $Z\in\zz$ with $\la Z, Z \ra =1$,  then the Lie algebra $\nn$ is a {\em Lie algebra  of type H  } \cite{K} (also known as $H$-type Lie algebras). Equivalently, the 2-step nilpotent Lie algebra $\nn$ is of {\em type H } if and only if   $$j(Z)^2=-\la Z, Z\ra Id, \qquad \mbox{ for every }Z\in\zz, $$
which is equivalent to 
$j(Z) j(\widetilde{Z})+ j(\widetilde{Z})j(Z) =-2\la Z, \widetilde{Z}\ra Id$, for $Z, \widetilde{Z}\in \zz$. 
By making use of this identity one can also prove that 
$$[X,j(Z)X]=\la X, X\ra Z$$
for every $X\in\vv$ and $Z\in\zz$.
	The real, complex and quaternionic Heisenberg algebras are examples of Lie algebras of type H.

\section{Trajectories for left-invariant magnetic fields}
In this section we write explicitly  the magnetic equation associated to a left-invariant Lorentz force on a 2-step nilpotent Lie group endowed with a left-invariant metric. We derive the solution for invariant Lorentz forces preserving the decomposition in Equation \eqref{decomp2}.

A {\em  Lorentz force} on a Riemannian manifold $(M,g)$ is a $(1,1)$-tensor $F:TM \to TM$,  which is skew-symmetric and such that the associated 2-form, $\omega_F$, is closed:
$$\omega_F(U,V)=g(FU,V), \quad \mbox{ for all }\quad  U,V\in \chi(M).$$
The closeness condition will impose restrictions to the 2-form as we will see in Section \ref{closedforms}.

On any Riemannian manifold $(M,g)$ with Levi-Civita connection $\nabla$, a solution curve  of the magnetic equation in \eqref{mageq}, namely $\gamma$,  has constant velocity. 
This follows easily from the next computation
$$\frac{d}{dt}g(\gamma'(t), \gamma'(t))=2g(\nabla_{\gamma'(t)}\gamma'(t), \gamma'(t))=2g(F\gamma'(t), \gamma'(t))=0.$$
However a reparametrization could not be a solution. In fact, assume $g(\gamma'(t), \gamma'(t))=E\neq 0$ and take $\tau(t)=\gamma(t/E)$, then $\tau'(t)=1/E\gamma'(t/E)$ so that one has 
$\nabla_{\tau'(t)}\tau'(t)= 1/E^2 F\gamma'(t/E)$. On the other side $F \tau'(t)=1/E F \gamma'(t/E).$

Now, let $\psi:M \to M$ denote an isometry and let $F$ denote a Lorentz force on the Riemannian manifold $M$. If $F$ commutes with the differential $d\psi$, then $\psi\circ \gamma$ is a magnetic trajectory whenever $\gamma$ it is.  In fact
$$ d\psi\nabla_{\gamma'}{\gamma'} = \nabla_{d \psi \gamma'}{d \psi \gamma'}=F \circ  d\psi \gamma'=d\psi \circ  F\gamma'.$$

Moreover take a non trivial number $r\in \RR$ and  let $\gamma$ denote a magnetic solution for the Lorentz force $F$. Write $\sigma(t)=\gamma(rt)$. Then $\sigma$ is magnetic solution for the Lorentz force $rF$:
$$\nabla_{\sigma'(t)}\sigma'(t)=r^2\nabla_{\gamma'(rt)}\gamma'(rt)=r^2F\gamma'(rt)=rF\sigma'(t).$$
Analogous computations prove the next lemma. 
\begin{lem} \label{lem2} Let $(M,g)$ denote a Riemannian manifold with Levi-Civita connection $\nabla$. Let $F$ denote a Lorentz force  on $M$. 
	\begin{enumerate}[(i)]
		\item If $\psi:M \to M$ is an isometry that preserves $F$: $d\psi \circ F = F \circ d\psi$, then $\psi \circ \gamma$ is a magnetic trajectory (for $F$).
		\item If $\gamma$ is a magnetic trajectory for the Lorentz force $F$, then $(\psi \circ\gamma)(rt)$ is a magnetic trajectory for the Lorentz force $d\psi \circ rF \circ d\psi^{-1}$, for  any  $r\in \RR$ and for any isometry $\psi:M \to M$.
	\end{enumerate}
\end{lem}

Let $(N,\la\,,\,\ra)$ denote a Lie group  equipped with a left-invariant metric  and let $\nabla$ be the corresponding Levi-Civita connection for the metric. Let $F$ be a skew-symmetric endomorphism on $TN$ which is left-invariant, that is $F\circ dL_p=dL_p \circ F$. In this situation the map $F$ is determined by its values on $T_eN\equiv \nn$. 
Assume that $F$ is a {\em left-invariant Lorentz force}, that is, it gives rise to a left-invariant closed 2-form. This means, in particular, that the map $F$ commutes with (the differential of) any translation on the left. Thus, for every magnetic trajectory $\gamma$ we have $\gamma'(t)=dL_{\gamma(t)}x(t)$ where  $x(t)$ is a curve at the Lie algebra. So,  the magnetic equation for $\gamma$  follows $$\nabla_{\gamma'(t)}\gamma'(t)= dL_{\gamma(t)} \nabla_{x(t)}x(t)=F \circ dL_{\gamma(t)} x(t)= dL_{\gamma(t)} \circ F x(t),  $$
where we denote also by $\nabla$ the $\RR$-bilinear map on $\nn$ determined by the Levi-Civita connection, sometimes called the Levi-Civita connection on the Lie algebra $\nn$.  From the equation above it is clear that magnetic trajectories are determined at the identity, furthermore, 
by curves $x:I \to \nn$ satisfying the equation

\smallskip

$x'(t)=\ad^*(x(t))(x(t)) + q F x(t),$

\smallskip

where $q\in \RR-\{0\}$ and  $\ad^*(x)$ denotes the adjoint of $\ad(x):\nn \to \nn$ with respect to the metric. This equation is achieved as one derives the Euler equation for a geodesic on a Lie group equipped with a left-invariant metric. Clearly,  for $F\equiv 0$ one obtains the Euler equation. 
\begin{lem} Let $(N, \la\,,\,\ra)$ denote a Lie group equipped with a left-invariant metric and left-invariant Lorentz force $F$. 
	The magnetic equation at the Lie algebra level is given by:
	\begin{equation}\label{magnetic-eq}
		x'(t)=\ad^*(x(t))(x(t)) + q F x(t), \quad \ad^*(x)\mbox{is the adjoint of $\ad(x)$ w.r.t. }\la\,,\,\ra.
	\end{equation}
	
\end{lem}

Assume now that $N$ denotes  a 2-step nilpotent Lie group with Lie algebra $\nn$.
Let $x:I \to \RR$ denote a curve on the 2-step nilpotent Lie algebra $\nn$.  Write $x(t)=v(t)+z(t)\in \vv \oplus \zz$ with respect to the orthogonal decomposition considered in \eqref{decomp2}.  For the  Lorentz force $F:\nn \to \nn$,  the curve $x$ satisfies  the magnetic equation \eqref{magnetic-eq} if and only if for the curves $v:I \to \vv$ and $z:I \to \zz$, where $I$ is a real interval, $I\subseteq \RR$,  it holds
\begin{equation}\label{exact-magnetic}
	v'(t)+z'(t)=j(z(t))v(t)+q Fv(t).
\end{equation}

By abuse, we shall say that a magnetic trajectory is {\em exact} if it is a solution of the magnetic equation for an ``{\em exact}'' magnetic field, that is, an exact 2-form  which  corresponds to a linear map on the Lie algebra of the form $F=j(\widetilde{Z})$ for some $\widetilde{Z}\in C(\nn)$. See more in Section \ref{closedforms}.

The following theorem characterizes the magnetic trajectories corresponding to exact left-invariant magnetic fields among the set of differentiable curves on $\nn$.

\begin{thm} Let $(N, \la\,,\,\ra)$ denote a  2-step nilpotent Lie group equipped with a left-invariant metric. Let  $y(t)=v(t)+z(t)$ denote a curve on the Lie algebra $\nn$. The following statements are equivalent:
	\begin{enumerate}[(i)]
		\item  $y(t)$ is a  magnetic trajectory for $F=j(\widetilde{Z})$;
		\item  $y(t)=x(t)-q\tilde{Z}$ where $x(t)$ is the geodesic with initial condition $X_0+Z_0+q\tilde{Z}$;
		\item the curve $y(t)=v(t)+z(t)$ in $\nn$  satisfies:
		\begin{enumerate}
			\item $z(t)$ is constant, namely $z(t)\equiv Z_0$,
			\item   $v'(t)=Rv(t)$, with  $R\in \mathfrak{so}(\vv)$, and 
			\item $R$ belongs to the image of $j$: $R\in Image(j)$. 
		\end{enumerate}  
		
	\end{enumerate}
	
\end{thm}
\begin{proof} Let $\gamma:I \to N$ be a curve on the Lie group $N$ such that $\gamma'(t)=dL_{\gamma(t)}y(t)$, for the curve $y(t)$ on the Lie algebra $\nn$.  By using the orthogonal decomposition of $\nn$ in  \eqref{decomp2}, any magnetic trajectory on $\nn$ must satisfy an equation of the  form
	\begin{equation}\label{2magnetic-alg}
		v'(t)=j(z(t))v(t)+qF_{\vv}v(t), \quad z'(t)=F_{\zz}v(t),
	\end{equation}
	where $F:\nn \to \nn$ is skew-symmetric with respect to $\la\,,\,\ra$ and it gives rise to a closed two form $\la F\cdot, \cdot \rangle$.
	
	(i) $\Longrightarrow$ (ii)	Assume that the Lorentz force $F$ corresponds to an exact $1$-form, that is, $F=j(\widetilde{Z})$. Let $y(t)$ denote the  magnetic trajectory on $\nn$ with initial condition $v(0)+z(0)=X_0+Z_0$. Then it satisfies the corresponding magnetic equation in \eqref{exact-magnetic}. 
	Let $x(t)$ denote the unique solution of the geodesic equation on $\nn$, 	with initial condition $X_0+ Z_0+q\widetilde{Z}$: 
	\begin{equation}\label{eq1}
		v'(t)=j(z(t))v(t), \quad z'(t)\equiv 0, 
	\end{equation}
	Then $x(t)-q \widetilde{Z}$ coincides with $y(t)$. In fact, let $\beta(t)=x(t)-q \widetilde{Z}$. The curve $\beta$ satisfies
	\begin{itemize}
		\item $\beta(0)=X_0+Z_0$
		\item $\beta'(t)=x'(t)$ with $x(t)=v(t)+z(t)$ satisfying \eqref{eq1}, which is
		
		$v'(t)=j(Z_0+q\widetilde{Z})v(t),\qquad  z'(t)\equiv 0 .$
		Thus $\beta$ also satisfies $v'(t)=j(Z_0+q\widetilde{Z})v(t)$. 
	\end{itemize} 
	Since $F=j(\widetilde{Z})$, the magnetic trajectory with initial condition $X_0+Z_0$ satisfies the Equation \eqref{2magnetic-alg}, which  takes the form
	$v'(t)=j(Z_0+q\widetilde{Z})v(t), \qquad z(t)\equiv 0$.  
	
	Thus by existence and uniqueness of the solution, it follows that  $\beta(t)$ coincides with $y(t)$.
	
	\smallskip
	
	(ii) $\Longrightarrow$ (iii) 
	Assume now that we have the curve on $\nn$ given by $y(t)=x(t)-q\widetilde{Z}$ where $x(t)$ is the geodesic with initial condition $X_0+Z_0+q\widetilde{Z}$. Thus,
	
	$y'(t)=v'(t)+z'(t)=j(Z_0+q\widetilde{Z})v(t)$.  Then $y(t)$ satisfies (iii). In fact, since $x(t)$ satisfies the geodesic equation on $\nn$, one has  $y'(t)= j(Z_0-\widetilde{Z}) v(t)$ and $z'(t)\equiv 0$. Clearly, if we denote by $R=j(Z_0-\widetilde{Z})$, then  $R \in Image(j)$.  	
	
	\smallskip
	
	(iii) $\Longrightarrow$ (i) 	 Let $y(t)$ denote a curve satisfying the conditions in (iii). We shall see that $y(t)$ is a magnetic trajectory. 
	
	Since $z'(t)\equiv 0$, then the curve satisfies $y(t)=v(t) + Z_0$ where $Z_0\in \zz$. On the hand, since $R\in Image(j)$, take  a non-trivial element in the commutator, $\widetilde{Z}$,  such that $R-j(Z_0)=j(\widetilde{Z})$. Then $v$ and $z$ satisfy the equations
	$$v'(t)=j(Z_0)v(t)+j(\widetilde{Z})v(t), \qquad z(t)=Z_0,$$
	which is the magnetic equation with $q\equiv 1$, and $F=j(\widetilde{Z})$. 
	This implies that $F$ is exact. And this finishes the proof.
\end{proof}
\begin{example}
	Consider the Heisenberg Lie algebra of dimension three,  $\hh_3$, let $x(t)$ denote a curve on $\nn$ such that $x(t)=v(t)+z(t)\in \vv \oplus \zz$. Take the Lorentz force given by $j(\rho e_3)$ for some $\rho\in \RR$. Assume $x(t)$ is a magnetic trajectory associated to $j(\rho e_3)$ with $x(0)=ae_1+be_2+ce_3$. As above, the curves $v(t)$ and $z(t)$ satisfy $v'(t)=(c+q\rho)j(e_3)v(t), \,z(t)=ce_3$, and this implies that the curve $x(t)$ is given by: 
	$$x(t)=e^{(c+q\rho)tj(e_3)}v_0+c e_3, \qquad \mbox{ where } v_0=u_1 e_1+ u_2 e_2.$$
	Notice that we assume $\rho \neq 0$ for the Lorentz force to be non-trivial. 
\end{example}
As above, let $F$ be a Lorentz force, that is, a $(1,1)$-tensor which is skew-symmetric on $TN$, left-invariant and it gives rise to a closed 2-form. Let $F_{\vv}$ and $F_{\zz}$ denote the projections of $F$ to $\vv$ and $\zz$ respectively, that is $F(W)=F_{\vv}(W)+F_{\zz}(W)\in \vv \oplus \zz$.

Let $\gamma: I \to N$ be a curve on the 2-step Lie group $N$ and write 
$\gamma(t)=\exp(X(t)+Z(t))$ where $X(t)\in \vv$ and $Z(t)\in \zz$. We shall derive the conditions for
$\gamma$ to be a magnetic trajectory with $\gamma (0)=e$ and $\gamma'(0) =X_0+Z_0$. By making use  of the derivation of the exponential map to obtain the tangent vector to $\gamma$, 
$\gamma'(t)=\frac{d}{dt}\exp(X(t)+Z(t)).$

Recall that $d \exp_{\xi}(A_{\xi})=\frac{d}{dt}|_{t=0} \exp(\xi + t A)= dL_{\exp \xi} (A+\frac12 [A, \xi])$ to obtain the formulas in the next lemma.

\begin{lem} Let $\gamma:I \to \RR$ be a curve on $(N,\la\,,\,\ra)$ given as $\gamma(t)=\exp(X(t)+Z(t))$, where $\exp:\nn \to N$ denote
	the usual exponential. Then $\gamma$ is a magnetic trajectory for the Lorentz force $F$ if and only if the curves on $\nn$ given by $X(t)$ and $Z(t)$ satisfy the following equations:
	\begin{equation}\label{magnetic-2step}
		\left\{ \begin{array}{rcl}
			X''-j(Z'+\frac12 [X', X])X' & = & q F_{\vv}(X'+ Z'+ \frac12 
			[X',X])\\
			Z''+\frac12 [X'',X] & = & qF_{\zz}(X'+ Z'+ \frac12 
			[X',X]).
		\end{array} \right.
	\end{equation}
\end{lem}

A particular case  of Equation \eqref{magnetic-2step} occurs if $F$ satisfies that $F_{\zz}\equiv 0$, since $F$ is skew-symmetric, one also has $F_{\vv}(\zz)\equiv 0$. 
In this case Equation \eqref{magnetic-2step} above reduces to 
\begin{equation}\label{magnetic-2stepred}
	\left\{ \begin{array}{rcl}
		X''-j(Z_0)X' & = & q F_{\vv}(X')\\
		Z'+\frac12 [X',X] & = & Z_0
	\end{array} \right.\end{equation}

\begin{prop} Let $F$ denote a left-invariant Lorentz force on $\nn$ and assume $F_{\zz}\equiv 0$. Let $\gamma(t)$ denote a magnetic trajectory  with $\gamma(0)=e$, $\gamma'(0)=X_0+Z_0\in \vv \oplus \zz$. Then
	\begin{equation}
		\gamma'(t)=dL_{\gamma(t)}e^{t(j(Z_0)+qF_{\vv})}X_0+ Z_0, \quad \mbox{ where } \gamma'(0)=X_0+Z_0\mbox{ for all } t\in \RR, 
	\end{equation}
	and $e^{t(j(Z_0)+qF_{\vv})}=\sum_{i=0}^{\infty}\frac{t^n}{n!}(j(Z_0)+qF_{\vv})^n$. 
\end{prop}
\begin{proof} Write $\gamma(t)=\exp(X(t)+Z(t))$ where $X(t)\subset \vv$ and $Z(t)\subset \zz$ for all $t\in\RR$. 
	Then 
	$$
	\begin{array}{rcl}
		\gamma'(t) & = & d \exp_{X(t)+Z(t)}(X'(t)+Z'(t))_{X(t)+Z(t)} \\
		& = & dL_{\gamma(t)}(X'+Z'+\frac12 [X',X])= dL_{\gamma(t)} X'+Z_0.
	\end{array}
	$$
	By integrating the first equation of \eqref{magnetic-2stepred} we obtain $X'(t)=e^{t(j(Z_0)+qF_{\vv})}X_0$, which finishes the proof. 
\end{proof}

Assume now that the Lorentz force satisfies $F(\vv)\subset \vv$ and $F(\zz)\subset \zz$. 
In this case Equation \eqref{magnetic-2step} reduces to 
\begin{equation}\label{magnetic-2stepred2}
	\left\{ \begin{array}{rcl}
		X''-j(Z'+\frac12 [X', X])X' & = & q F(X')\\
		Z''+\frac12 [X'',X] & = & qF( Z'+ \frac12 
		[X',X])
	\end{array} \right.
\end{equation}

To solve the above equations, it suffices to consider the solutions passing through the identity, that is $\gamma(0)=e$. In fact, this is possible since the metric is invariant by translations on the left, so as the Lorentz force $F$. 

Write $\gamma'(0)=X_0+Z_0\in \vv \oplus \zz$, and suppose  that $Z_0=Z_{0}^{\kappa}+Z_{0}^{\eta}$ where $Z_{0}^{\kappa}\in C(\nn)$ and $Z_{0}^{\eta}\in \ker(j)$, with respect to the orthogonal splitting $\zz=C(\nn)\oplus \ker(j)$.

The closeness condition for the Lorentz force $F$ implies that $F(\zz)\subseteq \ker(j)$. Since $F$ is a skew-symmetric map one has
$$\la F[U,V], Z\ra=-\la F(Z), [U,V]\ra =0,\quad \mbox{ for all }U,V\in \vv, Z\in \zz, $$
that gives  $F(C(\nn))\equiv 0$. So, in particular  $F(Z_{0}^{\kappa})=0$ and  $F(Z_{0}^{\eta})\in ker(j)$. 

By writing $Z(t)=Z_{\kappa}(t)+Z_{\eta}(t)\in C(\nn)\oplus \ker(j)$, the second equation in  \eqref{magnetic-2stepred2} is equivalent to 
\begin{equation}\label{eq11}
	\begin{array}{rcl}
		Z_{\kappa}''+\frac12 [X'',X] & = & 0 \quad \mbox{equivalently }\quad  Z_{\kappa}'+\frac12 [X',X] \equiv Z_0^{\kappa},\\
		Z_{\eta}'' & = & q F Z'_{\eta}.  
	\end{array}
\end{equation}
Since $Z'(t)=Z'_{\kappa}(t)+Z'_{\eta}(t)$ and $Z'_{\eta}\in \ker(j)$, the first equation in \eqref{magnetic-2stepred2} reduces to
\begin{equation}\label{eq22}
	X''-j(Z_0^{\kappa})X'  =  q F(X').\\
\end{equation}

Let $J:\vv \to \vv$ be the skew-symmetric map given by $J=j(Z_0^{\kappa})+qF_{\vv}$. And write $\vv =\vv_1 \oplus \vv_2$ where $\vv_1$ is the kernel of $J$ and $\vv_2$ is the orthogonal complement of $\vv_1$ in $\vv$. Note that $\vv_2$ is invariant by $J$ and the map $J$ is non-singular on $\vv_2$. 

Let $\{i\theta_1, -i \theta_1, \hdots, i\theta_N, -i\theta_N \}$ be the distinct non-null eigenvalues of $J$ and decompose $\vv_2$ as an orthogonal direct sum $\vv = \ww_1 \oplus \hdots \oplus \ww_N$ where every subspace $\ww_j$ is invariant by $J$ and $J^2=-\theta_j Id$ on $\ww_j$.  Write
$$X_0=X_1+X_2, \quad \mbox{ where } \quad X_j\in \vv_j, \quad j=1,2.$$
$$X_2=\sum_{j+1}^N \xi_j \quad \mbox{ where } \quad \xi_j\in \ww_j, \mbox{ for all }j=1, \hdots N.$$


Let $\ker(j)=\zz_1\oplus\zz_2$ be  a orthogonal splitting as vector spaces, where $\zz_1=\ker(j)\cap\ker(F_{\zz})$ and $\zz_2=\zz_1^{\perp}$ is its orthogonal complement. Thus the map $F$ leaves the subspace  $\zz_2$ invariant and the map $F$ is non-singular on $\zz_2$.   Write
$$Z_0^{\eta}=Z_1^{\eta}+Z_2^{\eta}, \quad \mbox{ where } \quad Z_j^{\eta}\in \zz_j, \quad j=1,2.$$

\begin{thm}\label{magnetictrayF0}  Let $(N, \la\,,\,\ra)$ denote a 2-step nilpotent Lie group with a left-invariant metric. Let $\gamma(t)=\exp(X(t)+Z(t))$ denote a magnetic trajectory through the identity  with initial condition $\gamma'(0)=X_0+Z_0\in \vv \oplus \zz$, where  $Z_0=Z_0^{\kappa}+Z_0^{\eta}\in C(\nn)\oplus ker(j)$.  Then, with respect to the notations above, one has
	\begin{enumerate}[(i)]
		\item $X(t)=t X_1 + (e^{tJ}-Id)J^{-1}X_2$, with $J=j(Z_0^{\kappa})+qF_{\vv}$ and
		\item $Z(t)=tZ_1(t)+Z_2(t)$ where
		\begin{enumerate}
			\item  $Z_1(t)=Z_0^{\kappa}+ Z_1^{\eta} + \frac12[X_1, (e^{tJ}+Id)J^{-1}X_2]+ \frac12 \sum_{j=1}^N[J^{-1}\xi_j,\xi_j ]$, 
			\item $Z_2(t)$ is a function of uniformly bounded absolute value:
			$$\begin{array}{rcl}
				Z_2(t) & = & \frac{1}{q}(e^{tqF}-Id)F^{-1}Z_2^{\eta} +[X_1, (e^{tJ}-Id)J^{-1}X_2] + \frac12 [e^{tJ}J^{-1}X_2,J^{-1}X_2] \\
				& & -\frac{1}{2} \sum_{i\neq j=1}^N \frac{1}{\theta_j^2-\theta_i^2}
				\left([e^{tJ} J\xi_i, e^{tJ}J^{-1}\xi_j]-[e^{tJ}\xi_i, e^{tJ}\xi_j]\right)	\\
				& & 	+\frac{1}{2} \sum_{i\neq j=1}^N \frac{1}{\theta_j^2-\theta_i^2}
				\left([ J\xi_i, J^{-1}\xi_j]-[\xi_i,\xi_j]\right).
			\end{array}$$
		\end{enumerate}
	\end{enumerate}
\end{thm}

\begin{proof} Clearly $X(t)$ is the solution of $X''(t)=JX'(t)$, which is equivalent to  Equation \eqref{eq22}. And in terms of  $\vv_i$ we shall have the solution $X(t)=tX_1+(e^{tJ}-Id)J^{-1}X_2$. Now, the equation on the center gives, on the one hand:
	
	$Z_{\eta}(t)= tZ_1^{\eta}+\frac{1}{q}(e^{tqF}-Id)F^{-1}Z_2^{\eta}.$ And on the other hand, 
	$$Z_{\kappa}'(t)=Z_0^{\kappa}-\frac12\left([X_1, (e^{tJ}-Id)J^{-1} X_2]+t[e^{tJ}X_2,X_1]+ [e^{tJ}X_2, (e^{tJ}-Id)J^{-1}X_2]\right).$$
	So, by computing  we have that:
	\begin{itemize}
		\item  the integral of $Z_0^{\kappa}$ is $t Z_0^{\kappa} + constant$, 
		
		\smallskip
		
		\item  the integral of $[X_1, (e^{tJ}-Id)J^{-1} X_2]$ is $[X_1, J^{-1} e^{tJ}J^{-1}X_2-tJ^{-1}X_2] + constant$, 
		
		\smallskip
		
		\item the integral of $ t [e^{tJ}X_2, X_1]$ is $t[e^{tJ}J^{-1}X_2, X_1] - [J^{-1} e^{tJ}J^{-1}X_2, X_1]+ constant$.  
		
		\smallskip
		
		\item the integral of $  [e^{tJ}X_2,J^{-1}X_2]$ is  $[e^{tJ}J^{-1}X_2,J^{-1}X_2]+ constant$.
	\end{itemize}
	Finally we need information about the integral of 
	$[e^{tJ}X_2, e^{tJ}J^{-1}X_2]$. 
	
	Notice that $$[e^{tJ}X_2, e^{tJ}J^{-1}X_2] = \sum_{i\neq j=1}^N[e^{tJ}\xi_i, e^{tJ}J^{-1}\xi_j]+ \sum_{i=1}^N [\xi_i, J^{-1}\xi_i]. $$ 
	The last terms come from the equality 
	$[\xi_i, J^{-1}\xi_i] \equiv  [e^{tJ}\xi_i, e^{tJ}J^{-1}\xi_i]\quad \mbox{ for all } i.$ 
	In fact, let $f_i(t)=[e^{tJ}\xi_i, e^{tJ}J^{-1}\xi_i]$ and derive with respect to $t$ to obtain that $f_i'(t)=0$ for all $t$. So that $f_i$ is a constant function, $f_i(t)=f_i(0)$. 
	
	Recall that $e^{tJ}J=Je^{tJ}$ for all $t\in \RR$. Notice that $J=-\theta_j^2J^{-1}$ on $\ww_j$. And make use of this information to prove that the functions $X(t)$ and $Z(t)$ are solutions of the magnetic equation \eqref{magnetic-2stepred2} with the required initial conditions. 
\end{proof}

\begin{rem}
	Let $(N, \la\,,\,\ra)$ denote a 2-step nilpotent Lie group equipped with a left-invariant metric. Let $F=j(\widetilde{Z})$ be a left-invariant Lorentz force on $N$. Indeed, the magnetic equation for $F$ has the form \eqref{magnetic-2stepred} with $F_{\vv}\equiv j(\widetilde{Z})$. Thus the solution is given in theorem above with $J=j(Z_0^{\kappa}+q\widetilde{Z})$. These solutions correspond to exact magnetic fields.
\end{rem}

\begin{example}
	Let $H_3$ denote the Heisenberg Lie group of dimension three endowed with the canonical left-invariant metric, that is, it makes the basis of left-invariant vectors $e_1, e_2, e_3$ to a orthonormal basis. 
	
	Let $\widetilde{Z}= \rho e_3$ such that one gets the Lorentz force $\rho j(e_3)$ giving rise to an exact magnetic field. 
	
	Let $\gamma(t)=\exp(x(t)e_1+ y(t) e_2 + z(t)e_3)$ denote a magnetic trajectory passing through the identity, for the Lorentz force $\rho j(e_3)$. Assume $\gamma'(0)=x_0e_1+y_0 e_2+z_0 e_3$. 
	
	In \cite{EGM} the authors obtain the explicit solutions as:
	\begin{itemize}
		\item if $z_0-\rho\neq 0$ then the solution is
		$$\left( \begin{matrix}
			x(t)\\
			y(t)
		\end{matrix}\right) = \frac{1}{z_0-\rho}\left( \begin{matrix}
			\sin(t(z_0-\rho)) & -1+\cos(t(z_0-\rho))\\
			1-\cos(t(z_0-\rho)) & \sin(t(z_0-\rho))
		\end{matrix}
		\right)$$
		and for $v_0=x_0e_1+y_0e_2$ set
		$$z(t)=z_0+\frac{||v_0||^2}{2(z_0-\rho)}t - \frac{||v_0||^2}{2(z_0-\rho^2)}\sin(t(z_0-\rho)).$$
		\item If $z_0=\rho$ the solution is 
		$\gamma(t)=\exp(t(x_0e_1+y_0 e_2+ z_0 e_3))$. 
		
	\end{itemize}
	
\end{example}

\begin{rem}\label{automtrayect}  Any orthogonal automorphism on the 2-step nilpotent Lie group $(N,\la\,,\,\ra)$ is determined at the Lie algebra level.  
	
	Let $\phi:\nn\to\nn$ be an orthogonal automorphism of $\nn$ and $0\neq r\in\RR$. If the curve $\gamma:I \to N$, given as $\gamma(t)=\exp(X(t)+Z(t))$ is a magnetic trajectory through the identity for the Lorentz force $F$, then by Lemma \ref{lem2} the curve $\gamma_{\phi,r}(t)=\exp(\phi(X(r t)+Z(r t)))$ is a magnetic trajectory through the identity for the Lorentz force $F_{\phi,r} = r\, \phi\circ F\circ \phi^{-1}$. Furthermore, whenever $F$ is of type I or type II then $F_{\phi,r}$ is of type I or type II, respectively (see section 4).  
	
	Observe that an orthogonal automorphism $\phi:\nn \to \nn$ verifies $\phi(\zz)=\zz$, $\phi(\vv)=\vv$ and
	$$\phi\circ J(Z) \circ \phi^{-1}=J({\phi(Z)})\quad \mbox{ for every }Z\in \zz.$$
\end{rem}

\section{Magnetic trajectories on Heisenberg Lie groups.}

In this section we study magnetic trajectories on Heisenberg Lie groups.

Let $\nn=\vv \oplus \zz$ (*) denote a 2-step nilpotent Lie algebra and let $F:\nn \to \nn$ be a skew-symmetric map giving rise to a closed 2-form, which is called a Lorentz force. We say that $F$ is of type I if it preserves the decomposition (*), while it is of type II if $F(\vv)\subseteq \zz$ and $F(\zz)\subseteq \vv$.

 We start by calculating the magnetic trajectories for Lorentz forces of type II on the Heisenberg Lie group of dimension three.

\subsection{Magnetic trajectories on the three dimensional Heisenberg Lie group}
 Let $\hh_3$ denote the three-dimensional Heisenberg Lie algebra. As seen in Example \ref{closedonh}, every left-invariant 2-form on the corresponding Lie group is closed. 
 
 Let $\la\,,\ra$ denote the canonical metric on $\hh_3$ (hence in $H_3$). Any skew-symmetric map of type II, $F:\nn\to\nn$  has  the form $F_U=-ad(U)+ad(U)^{T}$ for some $U\in\vv$, i.e. 
 $$F_U(V+Z)=[V,U] + j(Z) U, \quad \mbox{ for all } V+Z\in \vv \oplus \zz=\hh_3.$$
 
 In fact, assume that   such linear map $F$ has a matricial presentation in the basis $e_1, e_2, e_3$ as 	$$F=\left( \begin{matrix}
 	0 & 0  & - \beta\\
 0 & 0 & -\alpha \\
 	\beta & \alpha & 0
 \end{matrix}
 \right), \quad \alpha, \beta \in \RR. 
 $$
 Define the element $U\in \vv$ by  $U=-\alpha e_1+ \beta e_2$. Thus, one verifies  
 \begin{itemize}
 	\item 
 	 $[e_1,U]=\beta e_3$, $[e_2,U]=\alpha e_3$ and 
 	 \item $\ad(U)^t(Z)=j(Z)U$,  for $Z=z e_3$. 
 	\end{itemize}
 Since  $<j(Z)U, e_1>= \la Z, [U, e_1]\ra = -\beta z$, and $<j(Z)U, e_2>= \la Z, [U, e_2]\ra = -\alpha z$, it holds $F\equiv F_U$.

Let $\gamma(t)= exp(V(t)+Z(t))$ be a magnetic trajectory on the Heisenberg group through the identity element. Now, the magnetic equations in \eqref{magnetic-2step} for  the map $F_U$ can be written as 
\begin{equation}\label{magnetic-2step harmonic}
	\left\{ \begin{array}{rcl}
		V''-j(Z'+\frac12 [V', V])V' & = &  j(Z'+ \frac12 
		[V',V])U\\
		Z''+\frac12 [V'',V] & = & [V', U]
	\end{array} \right.
\end{equation}

One has the initial conditions $V(0)=(0,0)$, $Z(0)=0$ and  $V'(0)=V_0=x_0 e_1+ y_0 e_2$ and $Z'(0)=z_0e_3$. From the second equation one gets
$$Z'+\frac12 [V',V] = [V, U]+Z_0.$$

Replace in the first equation to obtain  

$$	V''-j([V, U]+Z_0)V'  =   j([V, U]+Z_0) U,$$
equivalently
\begin{equation}\label{magnetic-type2HeisSimp}
	V''-j([V, U]+Z_0) (V' + U)  = 0.
\end{equation}

Let $\phi:\hh_3\to\hh_3$ be an orthogonal automorphism of $\hh_3$ and $q\in\RR$. Then by Remark \eqref{automtrayect}, one gets
$$
\begin{array}{rcl}
	q\, \phi\circ F_U\circ \phi^{-1}(V+Z)=q\, \phi([\phi^{-1}(V),U] + j(\phi^{-1}(Z)) U)
&	=& q\, ([V,\phi(U)] +  j(Z)\phi (U)\\
&	= & F_{q\phi(U)}(V+Z).
\end{array}
$$
We saw in Example \ref{ortAutHeis} that for any orthogonal automorphism of $\hh_3$ there exists a transformation $A\in O(\vv)=O(2)$ such that  $\phi(V+Z)=A(V)+ det(A)Z$. Since the group $O(2)$ acts transitively on the spheres of $\vv$, for $U\neq 0$, there exist a real number  $q>0$ and an orthogonal automorphism $\phi$ of the Heisenberg Lie algebra $\hh_3$ such that $q\phi(U)=e_2$. By Lemma \ref{lem2} we may just compute the magnetic trajectory $\sigma(t)$ corresponding to a Lorentz force $F_U$, for $U=e_2$.

 Write   $V(t)=x(t)e_1+y(t)e_2$ and $U=e_2$.  Now, in usual  coordinates, the system \eqref{magnetic-type2HeisSimp} reduces explicitly to the system:
\begin{equation}\label{magnetic-type2HeisSimp2}
	\left\{ \begin{array}{rcl}
		x''(t)+(x(t)+z_0) (y'(t)+1) & = &  0\\
		y''(t)-(x(t)+z_0) x'(t)  & = & 0
	\end{array} \right.
\end{equation}

From the second equation we have
$$y''(t)=\left(\frac{x(t)^2}{2}+z_0 x(t)\right)'\Rightarrow y'(t)=\frac{x(t)^2}{2}+z_0x(t)+y_0.$$

Again replace in the first equation to get
\begin{equation}\label{eqmagneticHeisx}
	x''(t)+( x(t)+z_0) \left(\frac{ x(t)^2}{2}+z_0 x(t)+y_0+1\right) =  0
\end{equation}
This is a second order autonomous differential equation that has the following important property:
\begin{equation}
	\mbox{If } x(t) \mbox{ is a solution of Eq. \eqref{eqmagneticHeisx} then } x(t+c) \mbox{ and }x(-t+c)\mbox{ are also solutions for all }c\in\RR.
\end{equation}
In particular, whenever  $x(t)$ is a solution of Equation  \eqref{eqmagneticHeisx}  with initial condition $x'(0)=x_0$, then $x(-t)$ is a solution with initial condition $x'(0)=-x_0$. So,  we may assume that $x_0\geq 0$.

Let $h:\RR \to \RR$ be the function given by $h(x)= \frac{x^2}{2}+z_0x+y_0+1$ then Equation \eqref{eqmagneticHeisx} is rewritten as 
\begin{align*}
	x''(t)+h'(x(t))h(x(t))=0\\
	\mbox{ implying that }\quad   x'(t)x''(t)+x'(t)h'(x(t))h(x(t))=0\\
	\Rightarrow \qquad   \left(x'(t)^2 +h(x(t))^2\right)'=0\\
	\mbox{ and this says that } \quad  x'(t)^2 +h(x(t))^2 = x_0^2+(y_0+1)^2.
\end{align*}

Observe that this implies that $x'(t)^2+(y'(t)+1)^2=|V_1|^2=x_0^2+y_1^2$ is constant, where $V_1=V_0+e_2$ and $y_1=y_0+1$. Thus,  one gets
$$x'(t)^2 =|V_1|^2-h(x(t))^2.$$ 

If $x_0=0, y_0=-1$ or $x_0=z_0=0$, the  curve  $x(t)$ must be constant and the solution curve passing through the identity with initial conditions $V'(0)=V_0$, $Z'(0)= z_0 e_3$ is given by:

$x(t)\equiv 0$, $y(t)\equiv y_0 t$ and $z(t)\equiv z_0 t$.

We fix $|V_1|> 0$ and $x_0^2+z_0^2>0$. Suppose $x'(t)>0$ in some neighborhood of $0$ then

$$\frac{x'(t)}{\sqrt{|V_1|^2-h(x(t))^2}}  =1$$ 
implying that
$$\int_{0}^{x(t)}\frac{du}{\sqrt{|V_1|^2-h(u)^2}}  =t.$$
Then we must study the elliptic integral
$$\mathcal{E}(x)=\int_{0}^{x}\frac{du}{\sqrt{|V_1|^2-h(u)^2}}$$
and its inverse.

By the sustitution $v=u+z_0$ we get
\begin{align*}
	\mathcal{E}(x)=\int_{0}^{x}\frac{du}{\sqrt{|V_1|^2-\left(\frac{ u^2}{2}+z_0 u+y_1\right)^2}} = \int_{z_0}^{x+z_0}\frac{dv}{\sqrt{|V_1|^2-\left(\frac{v^2}{2}+y_1-\frac{z_0^2}{2}\right)^2}}\\
	=2\int_{z_0}^{x+z_0}\frac{dv}{\sqrt{(2|V_1|+2y_1-z_0^2+v^2)(2|V_1|-2y_1+z_0^2-v^2)}}
\end{align*}

Observe that $2|V_1|-2y_1+z_0^2> 0$ since $|V_1|\geq y_1$ and $x_0^2+z_0^2>0$. Then we have three cases to consider:  $2|V_1|+2y_1-z_0^2$ being  positive, negative or zero.

If  $2|V_1|+2y_1-z_0^2>0$, using formula $(7)$ on page 33 of \cite{Gr} we have
$$\mathcal{E}(x)=\frac{1}{2\sqrt{|V_1|}} \left(\mathrm{cn}^{-1}\left(\frac{z_0}{\sqrt{2|V_1|-2y_1+z_0^2}},k_1\right) - \mathrm{cn}^{-1}\left(\frac{x+z_0}{\sqrt{2|V_1|-2y_1+z_0^2}},k_1\right)\right) $$
where $\mathrm{cn}$ is the cosine amplitude, one of Jacobi's elliptic function, with  elliptic modulus $k_1=\sqrt{\frac{2|V_1|-2y_1+z_0^2}{4|V_1|}}$. This function is defined on $[-z_0-\sqrt{2|V_1|-2y_1-z_0^2} ,-z_0+\sqrt{2|V_1|-2y_1-z_0^2}]$.

Its inverse is 
\begin{equation}\label{solutionpositiveperiod}
	\Phi(t)=\sqrt{2|V_1|-2y_1+z_0^2}  \ \mathrm{cn}\left(C_0 -2\sqrt{|V_1|}\, t, k_1\right)-z_0
\end{equation}
where $C_0$ is such that $\Phi(0)=0$. The map $\Phi(t)$ is the solution of the differential equation \eqref{eqmagneticHeisx}, it is defined on $\RR$ and it is periodic with period $2\frac{K(k_1)}{\sqrt{|V_1|}}$ where $K(k)$ is the quarter period defined as
$$K(k)=\int_{0}^{\pi/2} \frac{1}{\sqrt{1-k^2 sin^2 \theta}}d\theta.$$

Analogously, if  $2|V_1|+2y_1-z_0^2<0$, we use formula $(9)$ on page 33 of \cite{Gr} to compute
$$\mathcal{E}(x)=\frac{1}{\sqrt{2|V_1|-2y_1+z_0^2}} \left(\mathrm{dn}^{-1}\left(\frac{z_0}{\sqrt{2|V_1|-2y_1+z_0^2}},k_2\right) - \mathrm{dn}^{-1}\left(\frac{x+z_0}{\sqrt{2|V_1|-2y_1+z_0^2}},k_2\right)\right) $$
where $\mathrm{dn}$ is the delta amplitude with modulus $k_2=\sqrt{\frac{4y_1-2z_0^2}{2|V_1|-2y_1+z_0^2}}$. In this case $\mathcal{E}(x)$ is defined on $ [-z_0-\sqrt{z_0^2-2y_1+2|V_1|},-z_0-\sqrt{z_0^2-2y_1-2|V_1|}]$ if $z_0<0$ or on $[-z_0+\sqrt{z_0^2-2y_1-2|V_1|},-z_0+\sqrt{z_0^2-2y_1+2|V_1|}]$ if $z_0>0$.

The corresponding solution of the differential equation is 
\begin{equation}\label{solnegativeperiodic}
	\Phi(t)=\sqrt{2|V_1|-2y_1+z_0^2}  \ \mathrm{dn}\left(C_1 -\sqrt{2|V_1|-2y_1+z_0^2}\, t, k_2\right)-z_0
\end{equation}
where $C_1$ is also the constant such that  $\Phi(0)=0$.
This solution is also periodic on $\RR$ and its period is $2\frac{K(k_2)}{\sqrt{2|V_1|-2y_1+z_0^2}}$.

Now if $2|V_1|+2y_1-z_0^2=0$, take the function 
\begin{align*}
	\mathcal{E}(x) 
	=2\int_{z_0}^{x+z_0}\frac{dv}{|v|\sqrt{4|V_1|-v^2}}.
\end{align*}
Thus, one has
\begin{align*}
	\mathcal{E}(x) 
	=\frac{sgn(z_0)}{\sqrt{|V_1|}}\left(\mathrm{artanh}\left(\sqrt{\frac{4|V_1|-z_0^2}{4|V_1|}}\right)-\mathrm{artanh}\left(\sqrt{\frac{4|V_1|-(x+z_0)^2}{4|V_1|}}\right)\right)	
\end{align*}
where the function $\mathrm{artanh}$ is the inverse hyperbolic tangent, i.e. $\mathrm{artanh}(x)=\frac{1}{2}ln\left(\frac{1+x}{1-x}\right)$. Observe that  the function $\mathcal{E}(x)$ is defined for $-z_0<x
\leq 2\sqrt{|V_1|}-z_0$ if $z_0>0$, and for $-2\sqrt{|V_1|}-z_0\leq x<-z_0$ if $z_0<0$.

So, if $z_0>0$ the solution of Equation \eqref{eqmagneticHeisx} is given by
\begin{equation}\label{solpositivenonperiodic}
	\Phi(t)=2\sqrt{|V_1|}\, \mathrm{sech}(C_2-\sqrt{|V_1|}t)-z_0
\end{equation}
and if $z_0<0$, it is given by
\begin{equation}\label{solnegativenonperiodic}
	\Phi(t)=-2\sqrt{|V_1|}\, \mathrm{sech}(\sqrt{|V_1|}t-C_2)-z_0,
\end{equation}
where $\mathrm{sech}$ is the hyperbolic secant and $C_2=\mathrm{artanh}\left(\sqrt{\frac{4|V_1|-z_0^2}{4|V_1|}}\right)$. These solutions are also defined on $\RR$ but they are not periodic.

All this is exposed  in the next result. 

\begin{thm}\label{thmsolutionmagHeisII}
	The magnetic trajectory $\sigma(t)=\exp(x(t)e_1 + y(t) e_2 +z(t))e_3)$ on the Heisenberg Lie group $H_3$ corresponding to a Lorentz force $F_{e_2}$,  with initial conditions $\sigma(0)=e\in H_3$ and $\sigma'(0)=x_0 e_1 + y_0 e_2 + z_0 e_3$ is explained below:
	\begin{itemize}
		\item If $x_0^2+ z_0^2 (y_0+1)^2\neq 0$,
		\begin{equation*}
			\left\{ \begin{array}{rcl}
				x(t)& = &  \Phi(t)\\ 
				y(t)& = & \displaystyle \int \limits_{0}^{t}  \left(\frac{\Phi(s)^2}{2}+z_0\Phi(s)+y_0\right)ds\\ 
				z(t)& = & \displaystyle \frac{1}{2}  \int \limits_{0}^{t}  \left(-\Phi'(s)y(s)+\Phi(s)(y'(s)+2)\right)ds +z_0 t\\
			\end{array} \right.
		\end{equation*}
		where the function $\Phi(t)$ is defined in Table \ref{table:1} if $x_0\geq0$ and we make use  of the corresponding $\Phi(-t)$ whenever $x_0<0$. 
		
		\item If $x_0= 0$ and $z_0 (y_0+1)= 0$, then
		\begin{equation*}
			\left\{ \begin{array}{rcl}
				x(t)& = & 0\\
				y(t)& = & y_0 t\\
				z(t)& = &  z_0 t.\\
			\end{array} \right.
		\end{equation*}
	\end{itemize}
\end{thm}

\renewcommand{\arraystretch}{1.9}
\begin{table}
	\begin{tabular}{|c|c|c|c| }
		\hline			
		Condition & $\Phi$ & Period & Image \\
		\hline\hline
		$z_0 < -\sqrt{2(|V_1|+y_1)}$ & 
		\eqref{solnegativeperiodic} & $\frac{2K(k_2)}{\sqrt{2|V_1|-2y_1+z_0^2}}$ & $\left [-z_0-\sqrt{z_0^2-2A},-z_0-\sqrt{z_0^2-2B}\right]$ \\ \hline
		$z_0 = -\sqrt{2(|V_1|+y_1)}$ & \eqref{solnegativenonperiodic} & non & $\left[-z_0-\sqrt{z_0^2- 2 A},-z_0\right)$ \\ \hline
		$|z_0| < \sqrt{2(|V_1|+y_1)}$ & \eqref{solutionpositiveperiod} & $\frac{2K(k_1)}{\sqrt{|V_1|}}$ & $\left[-z_0-\sqrt{z_0^2-2A},-z_0+\sqrt{z_0^2-2A}\right]$ \\ \hline
		$z_0 = \sqrt{2(|V_1|+y_1)}$ & \eqref{solpositivenonperiodic} & non & $\left(-z_0,-z_0+\sqrt{z_0^2-2A}\right]$ \\ \hline
		$z_0 > \sqrt{2(|V_1|+y_1)}$ & \eqref{solnegativeperiodic} & $\frac{2K(k_2)}{\sqrt{2|V_1|-2y_1+z_0^2}}$ & $\left[-z_0+\sqrt{z_0^2-2B},-z_0+\sqrt{z_0^2-2A}\right]$\\ \hline
		& & & where $A = y_1-|V_1| $, $B= y_1+|V_1|$\\ \hline
	\end{tabular}
	\caption{$|V_1|^2=x_0^2+(y_0+1)^2$ and $y_1=y_0+1$.}
	\label{table:1}
\end{table}

 \smallskip
 \begin{rem}
  See more details about elliptic functions and elliptic integrals in \cite{Gr}. There, one can see some applications of these integrals in many situations.

 \end{rem}
 
 \begin{rem}
 	Notice that any invariant Lorentz force on the Heisenberg Lie group $H_3$ will be of the form $F=j(\widetilde{Z})+F_U$ for some $U\in \vv$, in such way that whenever $\widetilde{Z}=0$, we are in the situation above. In the other hand, for $U=0$ we have an exact 2-form and $F=j(\widetilde{Z})$. 
 \end{rem}
 
 \medskip
 
{\bf Closed Magnetic trajectories on compact quotients.}

\medskip

 The aim now is the study of closed trajectories for the Lorentz forces of type II on Heisenberg nilmanifolds, that is $M=\Lambda\backslash H_3$, where $\Lambda$ denotes a cocompact lattice in $H_3$.  
 
 Let $(N, \la\,,\ra)$ denote a Lie group endowed with a left-invariant metric. Assume $\Lambda$ is a discrete subgroup of $N$ such that the quotient $\Lambda \backslash N$ is compact. A natural metric on the quotient is the one induced from $N$ and also a magnetic field on $\Lambda \backslash N$ is induced from the left-invariant magnetic field on $N$. 
 
\begin{defn}
	 Let $N$ be a simply connected nilpotent Lie group.  For any element  $\lambda \in N$ different from the identity, a curve  $\sigma(t)$ is called {\em $\lambda$-periodic} with period $\omega$ if $\omega \neq 0$ and for all $t \in \mathbb{R}$ it holds:
	$$
	\lambda \sigma(t)=\sigma(t+\omega).
	$$
	
\end{defn}

It is clear that whenever the subgroup  $\Lambda<N$ is a cocompact discrete subgroup, also called lattice,  for any element  $\lambda \in \Lambda$, a $\lambda$-periodic magnetic trajectory  will project to a smoothly closed magnetic trajectory under the mapping $N \rightarrow \Lambda \backslash N$.  Conversely, every closed magnetic trajectory $\sigma$ on $\Lambda \backslash N$ lifts to a $\lambda$-periodic magnetic trajectory on the Lie group $N$, if $\sigma$ is non-contractible, or directly lifts to a closed magnetic trajectory on $N$.

Choose an element $\lambda=exp(W_1+Z_1)\in N$. A magnetic trajectory $\sigma(t)=exp(V(t)+Z(t))$ is $\lambda$-periodic with period $\omega$, for  $\omega\neq 0$, if and only if the following equations are verified:
\begin{equation}\label{gamma-periodic2step}
	\left\{ \begin{array}{rcl}
		W_1+V(t) & = & V(t+\omega)\\
		Z_1+Z(t)+\frac{1}{2}[W_1,V(t)] & = & Z(t+\omega).
	\end{array} \right.
\end{equation}	
for all $t\in\RR$.
 
By assuming  that the curve $\sigma(t)=exp(V(t)+Z(t))$ is a magnetic trajectory for a left-invariant Lorentz force $F$ of type II on a 2-step nilpotente Lie group $(N, \la\,,\,\ra)$, then it satisfies the following system
\begin{equation*}\label{magnetic-2step}
	\left\{ \begin{array}{rcl}
		V''-j(Z'+\frac12 [V', V])V' & = & q F_{\vv}(Z'+ \frac12 
		[V',V])\\
		Z''+\frac12 [V'',V] & = & qF_{\zz}(V').
	\end{array} \right.
\end{equation*}	

From the second equation one gets:
\begin{equation*}
	Z'+\frac12 [V',V]  = qF_{\zz}(V)+Z_0.
\end{equation*}

Evaluating in $t+\omega$ and using \eqref{gamma-periodic2step} we get	
\begin{equation*}\label{Zmagneticeq2}
	Z'+\frac12 [W_1,V'] + \frac12 [V',W_1 +V]  = qF_{\zz}(W_1+ V)+Z_0.
\end{equation*}	
By substracting the last two equations, one  obtains the following condition for $W_1$:
\begin{equation}\label{Zmagnetic2}
	F_{\zz}(W_1)=0
\end{equation}

\begin{lem} \label{lambdaper}
	Let $\lambda=exp(W_1+Z_1)$ be any element in the 2-step nilpotent Lie group $(N, \la\,,\, \ra)$. If a left-invariant invariant Lorentz force $F$ of type II  admits a $\lambda$-periodic trajectory then $W_1\in \ker F$. %
	
\end{lem}

In the 3-dimensional Heisenberg, take the Lorentz force $F_{e_2}$. Clearly,  the kernel is given by 
\begin{equation*}
	F(W_1)=0 \Leftrightarrow W_1=ye_2, 
\end{equation*}
and more generally, for a non-trivial Lorentz force $F_U$  the kernel is $\ker F_U=span\{U\}$. 

Thus, the curve $\sigma(t)=(x(t),y(t),z(t))$ is $\lambda$-periodic for $\lambda=exp(y_1e_2+z_1e_3)$ if and only if
\begin{equation*}\label{gammaperiodic}
	\begin{array}{rcl}
		x(t) & = & x(t+\omega)\\
		y(t)+y_1 & = & y(t+\omega)\\
		z(t)+z_1-\frac12 y_1 x(t) & = & z(t+\omega)
	\end{array}
\end{equation*}
for all $t\in\RR$. So the function $x$ must be periodic and $\omega$ must be  multiple of the period of $x$. By using the expression of $y(t)$ given in Theorem \ref{thmsolutionmagHeisII},  we see that if $y_1= y(\omega)$, the second equation above holds. 
From the expression of  $z(t)$ in Theorem   \ref{thmsolutionmagHeisII} one has
\begin{align*}
	(z(t+w)-z(t)+\frac12 y_1 x(t))' = \\
	\frac{1}{2}\left(-x'(t+\omega)y(t+\omega)+x(t+\omega)(y'(t+\omega)+2) -  \left(-x'(t)y(t)+x(t)(y'(t)+2)\right) + y_1 x'(t)\right)\\
	= \frac{1}{2}\left(-x'(t)y(t+\omega)+x(t)(y'(t)+2) -  \left(-x'(t)y(t)+x(t)(y'(t)+2)\right) + y_1 x'(t)\right)\\
	= \frac{1}{2}x'(t)\left(-y(t+\omega) + y(t) +y_1\right)=0
\end{align*}

Then any magnetic trajectory $\sigma(t)=(x(t),y(t),z(t))$ where $x(t)$ is a periodic function with period $\omega$, is $\lambda$-periodic for $\lambda=exp(y(\omega)e_2+z(\omega)e_3)$. 

\begin{prop}\label{proplambdaperiod}
Let  $(H_3, \la\,,\,\ra)$ be the Heisenberg group equipped with its canonical metric, let   $F=F_{e_2}$ denote a Lorentz force of type II and let $\sigma$ be the magnetic trajectory through the identity such that $\sigma'(0)=x_0 e_1 + y_0 e_2 + z_0 e_3$. 
	\begin{itemize}
		\item If $z_0^2 \neq 2(\sqrt{x_0^2+(y_0+1)^2}+y_0+1)$, and $x_0^2+ z_0^2 (y_0+1)^2\neq 0$ then $\sigma$ is periodic or $\lambda$-periodic for $\lambda=\sigma(\omega)$ where $\omega$ is the corresponding period given in Table \ref{table:1}. 
		\item If $x_0^2+ z_0^2 (y_0+1)^2= 0$ then $\sigma$ is $\lambda$-periodic for $\lambda=exp(y_0 e_2+z_0 e_3)$. 
		\item If $z_0^2 = 2(\sqrt{x_0^2+(y_0+1)^2}+y_0+1)$, and $x_0^2+ z_0^2 (y_0+1)^2\neq 0$ then $\sigma$ is not $\lambda$-periodic for any $\lambda$.  
		\end{itemize}
\end{prop}

\subsection{ Magnetic trajectories for Lorentz forces of type I}

Here we shall study some magnetic trajectories in the Heisenberg Lie group of dimension five, $H_5$ (see Example \ref{ExHeis}), where we take a Lorentz force of type I which is not exact. 

\smallskip

Let $F$ be  a Lorentz force  of type I on $H_5$ such that $F j(Z)=j(Z) F$.  Then $F$ can be identified with a skew-hermitian matrix on $\RR^4\simeq \vv$ which has purely imaginary eigenvalues $i\mu_1,\,i\mu_2$ and it is diagonalizable by $S\in U(2)$. This means that there is $S\in O(4,\RR)$ such that $Sj(Z)S^{-1} = j(Z)$ and there is a basis of $\vv$ of the form $U_1, V_1, U_2, V_2$ giving the following matricial presentation $$S F S^{-1}=\left(\begin{matrix}
	0 &  -\mu_1 &  &  \\
	\mu_1 & 0   &  &  \\
	 &  &  0  & -\mu_2 \\
	 & &  \mu_2 & 0
\end{matrix}\right).$$
Observe that $F$ is exact if and only if $\mu_1=\mu_2$, while it is harmonic if $\mu_1=-\mu_2$. 

Since $Sj(Z)S^{-1} = j(Z)$ we have that
$$S(z_0 j(Z)+F)S^{-1} =\left(\begin{matrix}
	0 &  -z_0-\mu_1 &  &  \\
	z_0+\mu_1 & 0   &  &  \\
	&  &  0  & -z_0-\mu_2 \\
	& &  z_0+\mu_2 & 0
\end{matrix}\right). $$

Now, take the curve $\widetilde{\sigma}(t)=\exp(SV(t)+ Z(t))$, for $S\in U(2)$ as above. Then $||\sigma'(t)||=||\widetilde{\sigma}'(t)||$ for all $t$, and from the equations above we have that the curves  $SV(t)$ and $Z(t)$ in the Lie algebra $\nn$,  satisfy:
$$SV''=S(F+z_0 j(Z))S^{-1}SV'=\widetilde{J}SV',\quad  Z'+\frac12 [SV',SV]=z_0 Z.$$
This means that finding solutions to these last equations gives magnetic trajectories for the original magnetic equation. 
Moreover, the curve $\widetilde{\sigma}$ is closed or periodic if equivalently  the curve $\sigma$ satisfies this. 

Solutions to the magnetic equations above for the Lorentz force $\widetilde{J}= S(z_0 j(Z)+F)S^{-1}$ have the form $\widetilde{\sigma}(t)=\exp(SV(t)+ z(t)Z)$ for $SV(t)\in \vv$; explicitly:

\begin{itemize}
\item for $-z_0=\mu_1=\mu_2$:
$$\begin{array}{rcl}
SV(t) & = &  SV_0t, \quad V_0= x_1^0 X_1+ y_1^0 Y_1+ x_2^0 X_2+ y_2^0 Y_2\\
z(t) & = & z_0 t. 
\end{array}$$
\item for $z_0 + \mu_1 \neq 0$ and $ z_0 + \mu_2\neq 0$:
$$ \begin{array}{rcl}
	 SV(t) & = & (u_1(t), v_1(t), u_2(t), v_2(t))\quad \mbox{ with }\\
	 u_1(t) & = & \frac{\tilde{x}_1^0}{z_0 + \mu_1}\sin\left((z_0+\mu_1)t\right) - \frac{\tilde{y}_1^0}{z_0 + \mu_1}(1- \cos\left((z_0+\mu_1)t\right)), \\
	 v_1(t) & = & \frac{\tilde{x}_1^0}{z_0 + \mu_1}(1- \cos\left((z_0+\mu_1)t\right))+\frac{\tilde{y}_1^0}{z_0 + \mu_1}\sin\left((z_0+\mu_1)t\right),\\
	 u_2(t) & = &   \frac{\tilde{x}_2^0}{z_0 + \mu_2}\sin\left((z_0+\mu_2)t\right) - \frac{\tilde{y}_2^0}{z_0 + \mu_2}(1- \cos\left((z_0+\mu_2)t\right)),\\
	 v_2(t) & = &  \frac{\tilde{x}_2^0}{z_0 + \mu_2}(1-\cos\left((z_0+\mu_2)t\right))+\frac{\tilde{y}_2^0}{z_0 + \mu_2}\sin\left((z_0+\mu_2)t\right),\\
	 z(t) & = & (z_0 + \frac{1}2 \frac{||\widetilde{V}_1^0||^2}{z_0 + \mu_1} + \frac{1}2 \frac{||\widetilde{V}_2^0||^2}{z_0 + \mu_2})t- \frac12 \displaystyle \sum_{i=1}^2 \frac{||\widetilde{V}_i^0||^2}{(z_0+\mu_i)^2}  \sin (z_0+\mu_i)t,
	\end{array}
$$
where $\widetilde{V}_i^0=(\widetilde{x}_i^0, \widetilde{y}_i^0)$.

\item For $-z_0=\mu_2\neq \mu_1$:
$$ \begin{array}{rcl}
	 SV(t) & = & (u_1(t), v_1(t), u_2(t), v_2(t))\quad \mbox{ with }\\
	 u_1(t) & = & \frac{\tilde{x}_1^0}{z_0 + \mu_1}\sin\left((z_0+\mu_1)t\right) - \frac{\tilde{y}_1^0}{z_0 + \mu_1}(1- \cos\left((z_0+\mu_1)t\right)), \\
	 v_1(t) & = & \frac{\tilde{x}_1^0}{z_0 + \mu_1}(1- \cos\left((z_0+\mu_1)t\right))+\frac{\tilde{y}_1^0}{z_0 + \mu_1}\sin\left((z_0+\mu_1)t\right),\\
	 u_2(t) & = &   \tilde{x}_2^0 t,\\
	 v_2(t) & = &  \tilde{y}_2^0 t,\\
	 	 z(t) & = & (z_0 + \frac{1}2 \frac{||\widetilde{V}_1^0||^2}{z_0 + \mu_1})t- \frac12  \frac{||\widetilde{V}_1^0||^2}{(z_0+\mu_1)^2}  \sin (z_0+\mu_1)t,
	\end{array}
$$
where $\widetilde{V}_1^0=(\widetilde{x}_1^0, \widetilde{y}_1^0)$. 

\end{itemize} 
The values $\tilde{x}_1^0, \tilde{y}_1^0, \tilde{x}_2^0, \tilde{y}_2^0$ are the transformed by $S$ of the initial conditions for $V(t)$.

Recall that the {\em  energy} is $E(X)=\frac12 \la X, X \ra$, for any $X\in \nn$. In particular, the {\em  energy of a curve} $\sigma$ is the energy of $\sigma'(0)$. 

\begin{prop}
	If the map $F:\nn \to \nn$ does not correspond to an exact form,  there is a periodic magnetic trajectory with energy $E$ for every $E\geq0$.
\end{prop}

\begin{proof}

The condition of  the map $F$ giving rise to a non-exact closed 2-form, says that the eigenvalues satisfy $\mu_1\neq \mu_2$ and we can assume without losing generality that $\mu_1 < \mu_2$. Note that to have periodic trajectories, one needs to consider the linear part of the coordinate $z(t)$ of the solution above.
	
	Take $z_0$ such that $-\mu_2<z_0<-\mu_1$ then the line 
	$$ \frac{1}2 \frac{x}{z_0 + \mu_1} + \frac{1}2 \frac{y}{z_0 + \mu_2}=-z_0$$ has positive slope and the set of points on the line with non-negative coordinates is not bounded.
	More precisely, for $E\geq max \{\frac{|\mu_1|^2}{2},\frac{|\mu_2|^2}{2}\}$ we have that $2E-z_0^2>0$ and 
	$$2E-z_0^2\geq \mu_i^2-z_0^2>\mu_i^2-z_0^2-(\mu_i+z_0)^2=-2z_0(z_0+\mu_i)$$
	for $i=1,2$. So the system
	\begin{equation}\label{sistcerradasHeis}
		\left\{ \begin{array}{rcl}
			\frac{1}2\frac{x}{z_0 + \mu_1} + \frac{1}2 \frac{y}{z_0 + \mu_2}&= &-z_0\\
			x+y&= &2E-z_0^2.
		\end{array} \right.
	\end{equation}
	has a solution with $x\geq0$ and $y\geq0$. Take  $\widetilde{V}_1^0, \widetilde{V}_2^0\in \vv$ such that $||\widetilde{V}_1^0||^2=x,||\widetilde{V}_2^0||^2=y$ is the solution of the system \eqref{sistcerradasHeis}. Then every magnetic trajectory with initial condition $(\widetilde{V}_1^0$,$\widetilde{V}_2^0,z_0)$ is periodic  and has energy $E=\frac{1}{2}(z_0^2+||\widetilde{V}^0_1||^2+||\widetilde{V}^0_2||^2)$.
	
	Now suppose $E< max \{\frac{|\mu_1|^2}{2},\frac{|\mu_2|^2}{2}\}=\frac{|\mu_1|^2}{2}$ (the case $max \{\frac{|\mu_1|^2}{2},\frac{|\mu_2|^2}{2}\}=\frac{|\mu_2|^2}{2}$ is analogous). If there is a periodic magnetic trajectory with  $\widetilde{V}_2^0=0$ and energy $E$ then
	$$z_0 + \frac{1}2 \frac{||\widetilde{V}_1^0||^2}{z_0 + \mu_1} + \frac{1}2 \frac{||\widetilde{V}_2^0||^2}{z_0 + \mu_2}=0 \Leftrightarrow (z_0+\mu_1)^2  + 2E = \mu_1^2 \Leftrightarrow z_0 = -\mu_1 \pm \sqrt{\mu_1^2 - 2E}.$$  
	
	Observe that $(-\mu_1 - \sqrt{\mu_1^2 - 2E})(-\mu_1 + \sqrt{\mu_1^2 - 2E})=2E$. Since $z_0^2\leq 2E$, the only possibility is $z_0 = -\mu_1 +sgn(\mu_1) \sqrt{\mu_1^2 - 2E}$. Then taking $\widetilde{V}_1^0$ such that $||\widetilde{V}_1^0||^2 = 2E - z_0^2$ we have that the corresponding magnetic trajectory is periodic and has energy $E$.
	
	 Note that in the special case that $z_0=-\mu_2$ (when $2E = \mu_1^2-(\mu_1-\mu_2)^2$) this works as well since $\widetilde{V}_2^0=0$ and $u_2(t) = v_2(t)=0$ for all $t$.	
	
\end{proof}

\begin{rem}
Notice that for exact 2-forms in the Heisenberg Lie group there exist closed 
magnetic trajectories only for sufficiently small energy levels and the $\lambda$-periodic trajectories has energy bounded below by the Ma\~n\'e  critical value \cite{EGM}. The result above and Proposition \ref{proplambdaperiod} shows that for some non-exact magnetic fields there are periodic and $\lambda$-periodic  magnetic geodesics trajectories at every energy level.   
\end{rem}

\end{document}